\documentclass[12pt]{article}
\usepackage{amsmath,amssymb,amsthm}
\usepackage{amsfonts}
\usepackage{graphicx}
\usepackage{color}

\newtheorem{Theorem}{Theorem}[section]
 
\newtheorem{Proposition}[Theorem]{Proposition}

\theoremstyle{definition}
\newtheorem{Definition}[Theorem]{Definition}

 \def\C{\mathbb{C}} 
\def\ei{\mathbf{i}}
\def\ej{\mathbf{j}}
\def\ek{\mathbf{k}}

\def\H{\mathbb{H}}
\def\M{\mathcal{M}} 
\def\R{\mathbb{R}}

\DeclareMathOperator{\Ker}{Ker}

\DeclareMathOperator{\Sc}{Sc} 

\newcommand{\pic}[4]{\vspace{1ex}\setlength{\unitlength}{.05\textwidth}
\begin{picture}(0,0)(0,0)
\put(#2){\includegraphics[#4]{#1.jpg}} 
\end{picture}\vspace*{#3cm}\vspace{1ex}}

\begin{document}

\begin{center}

  \textbf{\Large Quaternionic metamonogenic functions\\ in the unit disk}
\\[2ex]
  {J.~Morais$^{a}$ and R.~Michael Porter$^{b}$}\\[1ex]
  {$^{a}$Department of Mathematics, ITAM, R\'io Hondo \#1,\\
    Col. Progreso Tizap\'an, 01080 Mexico City, Mexico. \\ E-mail:
    joao.morais@itam.mx \\ $^{b}$Department of Mathematics,
    CINVESTAV-Quer\'etaro,\\ Libramiento Norponiente \#2000, Fracc. Real
    de Juriquilla.\\   Santiago de Quer\'etaro, C.P. 76230 Mexico}
\end{center}


 \begin{abstract}
   We construct a set of quaternionic metamonogenic functions (that is, in $\Ker(D+\lambda)$ for diverse $\lambda$)
     in the unit disk, such that every metamonogenic function is
     approximable in the quaternionic Hilbert module $L^2$ of the
     disk. The set is orthogonal except for the small subspace of elements of orders zero and one. These functions are used to express time-dependent solutions of the
     imaginary-time wave equation in the polar coordinate system.
\end{abstract}
 
\textbf{Keywords:}
Quaternionic analysis, Bessel functions, quaternionic functions, Moisil-Teodorescu operator.


\section{Introduction}

We consider the first-order partial differential quaternionic operator
(sometimes called the Moisil-Teodorescu operator) $D+\lambda$
($\lambda \in \R\setminus\{0\}$) in planar domains, where
\begin{equation} \label{eq:defD}
D = \frac{\partial}{\partial x} \ei + \frac{\partial}{\partial y} \ej.
\end{equation}
Here $\ei$ and $\ej$ are two of the three basic quaternionic units,
and the Dirac operator $D$ acts on the left on smooth
quaternion-valued functions of a complex variable $z=x+iy$. This is a
special case of Clifford-type operators for which there is a vast
literature covering functions defined on spaces of diverse
dimensions, beginning with W.\ Hamilton \cite{WRH1853}, continued by
R.\ Fueter \cite{Fueter1940,Fueter1949} and followed by many others
including
\cite{Abul-EzConstales1990,BrackxDelangheSommen1982,CacaoFalcaoMalonek2011,Delanghe2007,Delanghe2009,GuerlebeckHabethaSproessig2008,GuerlebeckHabethaSproessig2016,MoraisHabilitation2019}.

Functions defined in $\R^2$ taking values in Clifford algebras of
dimension $>2$ have been relatively less investigated. The case
$\R^2 \to\R^4 \cong \H$ was considered in \cite{LuPeRoSha2013}, where
a detailed investigation of $D+\lambda$ was carried out for
quaternion-valued functions in the particular situation of elliptical
domains. In \cite{MoraisPorter2021}, the authors constructed a
one-parameter family of reduced quaternion-valued ($\R^3$-valued)
functions of a pair of real variables lying in an ellipse, termed
$\lambda$-metamonogenic Mathieu functions. Returning to the context
$\R^2 \to\H$, we consider here the case
of functions defined in the unit disk employing Bessel functions of
the first kind in place of Mathieu functions. We produce a set of
metamonogenic functions (that is, in $\Ker(D+\lambda)$ for diverse
$\lambda$), which is orthogonal in the unit disk for orders $\ge2$ and
in a certain sense complete in $\Ker(D+\lambda)\cap L^2$ for every
$\lambda$. As an application, in the final section we use these
functions to express time-dependent solutions of the imaginary-time
wave equation in the disk.


\section{Preliminaries\label{sec:preliminaries}} 

\subsection{Metamonogenic functions}


We consider the quaternionic operator $D$ defined by
\eqref{eq:defD}. This is interpreted as follows, in fairly standard
notation and terminology, in which $z=x+iy$ is a complex number, and a
quaternion \cite{K2003,MoraisGeorgievSproessig2014} is notated as
$a =a_0+ a_1\ei+ a_2\ej+ a_3\ek$. Here $a_m\in\R$ and $\ei,\ej,\ek$
are the quaternionic imaginary units satisfying
$\ei^2 = \ej^2 = \ek^2 = \ei \ej \ek = -1$. The set of real
quaternions $\H=\H(\R)$ is naturally identified with $\R^4$, which
determines the usual component-wise addition and also induces the
absolute value on $\H$.  Thus $D$ acts on $\H$-valued functions
\begin{eqnarray*}
f(x,y) = f_0 (x,y) + f_1(x,y) \ei +
     f_2(x,y) \ej + f_3(x,y) \ek,
\end{eqnarray*}
defined in domains in the complex plane $\C$ applying the quaternionic
multiplication rules, in principle, on the left or right, giving $Df$
or $fD$. We will only consider the operator acting from the left, as
the other case is analogous. 

Let $\Omega$ be a domain in $\mathbb{R}^2$ (open and connected). Let
$L^2(\Omega)=L^2(\Omega,\H)$ denote the space of all $\H$-valued
functions $f\colon\Omega\to\H$ such that the components
$f_m$ ($m=0,1,2,3$) are in the usual $L^2(\Omega,\R)$. It is
easily seen that $L^2(\Omega)$ is naturally a right $\H$-linear module
and admits the $\H$-valued right inner product
\begin{eqnarray} \label{eq:InnerProduct}
\langle f, g \rangle_\H =
  \int \!\!\!\int_{\Omega} \;  \overline{f(x,y)} \, g(x,y) \,dx\,dy
\end{eqnarray}
for $f, g \in L^2(\Omega)$.
Thus $L^2(\Omega)$ is a quaternionic right Hilbert module with the
associated norm
$\|f\|_2 = \langle f, f\rangle^{1/2}_\H = \langle f, f
\rangle^{1/2}_\R$, where
$\langle f, g \rangle_\R = \Sc \langle f, g \rangle_\H$ coincides
with the usual $L^2$-norm for $f$, viewed as an $\R^4$-valued function
in $\Omega$ \cite{GS1989,GS1997}.

For functions taking values in the $2$-dimensional subspace
$\R\ei\oplus\R\ej\subseteq\H$, $D$ echoes the classical Cauchy-Riemann operator
$2\partial/\partial \overline{z}=\partial/\partial x+i\partial/\partial y$, but it sends such
functions to the complementary subspace $\R\oplus\R\ek\subseteq\H$.

As usual, $\Delta = \partial^2/\partial x^2 + \partial^2/\partial y^2$
will denote the Laplace operator in $\R^2$.
\begin{Definition} \label{Main_definitions} Let $\Omega\subseteq\R^2$. Given
  $\lambda \in \R \setminus \{0\}$, a function  $f\in C^2(\Omega,\R)$ is
  said to be \textit{$\lambda$-metaharmonic} when
  \[ (\Delta+\lambda^2)f=0 .\]
When $f\in C^1(\Omega,\H)$ and
\[ (D+\lambda)f=0 ,\] one says that $f$ is \textit{left
  $\lambda$-metamonogenic}.
 \end{Definition}
 
We thus have the spaces of left $\lambda$-metamonogenic functions
\begin{equation*}
\M(\Omega;\lambda) = \Ker (D + \lambda) \subseteq C^1(\Omega,\H)
\end{equation*}
and
\begin{equation*}
\M_2(\Omega;\lambda) = \M(\Omega;\lambda) \cap
L^2(\Omega).
\end{equation*}
It is well known \cite{Vekua2013} that metaharmonic functions are of
class $C^\infty$, and so by the following factorization of the
Laplacian via $D$ (cf.\ \cite[Section CVII]{WRH1853} and
\cite{Gurl1986}), metamonogenic functions are of class $C^\infty$
also.

\begin{Proposition}  \label{prop:factorhelmholtz}
$(D+\lambda)(D-\lambda) = -(\Delta+\lambda^2)$ for $\lambda \in \R \setminus \{0\}$.
\end{Proposition}
 
In polar coordinates  $x=\rho\cos\theta$, $y=\rho\sin\theta$, one has
\begin{equation}  \label{eq:polar}
  D_{r,\theta} = \bigg(\cos\theta\frac{\partial}{\partial\rho } +
  \frac{\sin\theta}{\rho}\frac{\partial}{\partial\theta}\bigg)\ei
 +  \bigg(\sin\theta\frac{\partial}{\partial\rho } -
  \frac{\cos\theta}{\rho}\frac{\partial}{\partial\theta}\bigg)\ej,
\end{equation}
and the Helmholtz operator in polar coordinates is
\begin{equation}
  \Delta_{\rho,\theta} + \lambda^2 = \frac{\partial^2}{\partial\rho ^2} + \frac{1}{\rho} \frac{\partial}{\partial\rho } + \frac{1}{\rho^2} \frac{\partial^2}{\partial \theta^2} + \lambda^2.
\end{equation}

From Proposition \ref{prop:factorhelmholtz}, it is clear that
the components of any $f\in\M(\Omega;\lambda)$ are $\lambda$-metaharmonic.
The equation $(D+\lambda)f=0$ is equivalent to the system
of partial differential equations
\begin{align}  
  \frac{\partial f_1}{\partial x} +  \frac{\partial f_2}{\partial y}
   - \lambda f_0 &=0, \nonumber \\
  -\frac{\partial f_0}{\partial x} -  \frac{\partial f_3}{\partial y}
   - \lambda f_1  &=0, \nonumber\\
  \frac{\partial f_3}{\partial x} -  \frac{\partial f_0}{\partial y}
   - \lambda f_2  &=0, \nonumber\\
  -\frac{\partial f_2}{\partial x} +  \frac{\partial f_1}{\partial y }
  - \lambda f_3 &=0.  \label{eq:metasystem}
\end{align}

From this the following is immediate.
\begin{Proposition} \label{prop:f1f2}
  Given any two scalar $\lambda$-metaharmonic functions $f_1$, $f_2$ in any domain
  $\Omega\subseteq\C$, there are unique
  ($\lambda$-metaharmonic) functions $f_0$, $f_3$ such that $f_0+f_1\ei+f_2\ej+f_2\ek$ is
  $\lambda$-metamonogenic.
\end{Proposition}

Indeed, take
\begin{align*}
   f_0  = \frac{1}{\lambda}\big(\frac{\partial f_1}{\partial x} +
                 \frac{\partial f_2}{\partial y}\big) ,\quad
   f_3  = \frac{1}{\lambda}\big(\frac{\partial f_1}{\partial y} -
                   \frac{\partial f_2}{\partial x} \big),
\end{align*}
and observe that the system \eqref{eq:metasystem} is satisfied.
Similarly, given $f_1$, $f_2$ one has unique functions $f_0$, $f_3$
completing the components of a $\lambda$-metamonogenic function. One
may think of the relationship of pairs $(f_1,f_2)$ and $(f_0,f_3)$ as
a generalization of the notion of harmonic conjugates.


\section{Quaternionic metamonogenic functions} \label{Quaternionic_functions}

This section introduces a family of $\lambda$-metamonogenic
functions in the real Hilbert space $L^2$ of the unit disk, which is
the object of study of this paper.
 
The factorization of Proposition \ref{prop:factorhelmholtz} suggests that
quaternionic $\lambda$-meta\-mo\-no\-genic functions should play a
role for the Laplace operator $\Delta$, similar to the usual metaharmonic
functions in two variables for the corresponding Helmholtz operator
\cite{Vekua1968}.


\subsection{A class of $\lambda$-metamonogenic functions\label{subsec:RQMfunctions}}

First we define a continuous family of quaternionic metamonogenic
functions.  Let $J_n(z)$, $z\in\C$ denote the $n$-th Bessel function
of the first kind, $n=0,1,2,\dots$  \cite{JeffreyDai2008}.  We recall
that
\begin{align} \label{eq:bessel0}
  J_0(0) &=1,\quad J_n(0)=0\;\, (n\not=0), \nonumber \\
  J_1'(0) &= \frac{1}{2},\quad J_n'(0)=0\;\, (n\not=1).
\end{align} 
 
\begin{Definition} \label{Definition_RQM functions} Let $n \geq 0$ and
  $\lambda\in\R\setminus\{0\}$. For $z=x+iy=\rho e^{i\theta}\in\C$, the
  \textit{$n$-th standard $\lambda$-metamonogenic function} is
\begin{align*}
  F_n[\lambda](z) &= \big( J_n'(\lambda\rho ) +
    \frac{n}{\lambda\rho }J_n(\lambda\rho ) \big) \cos(n-1)\theta  \\
    & \quad\quad + \big(J_n(\lambda\rho )\cos n\theta\big)\, \ei \ +
   \  \big(J_n(\lambda\rho )\sin n\theta\big) \, \ej \\
  & \quad\quad- \big(( J_n'(\lambda\rho ) +  \frac{n}{\lambda\rho }J_n(\lambda\rho ) ) \sin(n-1)\theta \big)  \, \ek  \quad {\rm if} \ z\not=0
\end{align*}
and for $z=0$ the limiting value,
\begin{align*}
F_0[\lambda](0) = \ei,\quad F_1[\lambda](0) = 1 ,\quad   F_n[\lambda](0) = 0 \quad (n \geq2).
\end{align*}
\end{Definition}
In particular, $F_0[\lambda](z) = -J_1(\lambda\rho ) \cos\theta + J_0(\lambda\rho ) \, \ei
\ - J_1(\lambda\rho ) \sin\theta \, \ek$ because of the second of the recurrence relations
\begin{align}  
   \frac{2n}{z}J_n(z) &= J_{n-1}(z) + J_{n+1}(z), \nonumber \\
   2J_n'(z) &= J_{n-1}(z) - J_{n+1}(z), \label{eq:besselrecurrence1}
\end{align}
with $J_{-1}(z) = -J_1(z)$.

Note that the $\ei$ and $\ej$ components of $ F_n[\lambda](z)$ are
the classical solutions $J_n(\lambda\rho )\cos n\theta$,
$J_n(\lambda\rho )\sin n\theta$ for the Helmholtz equation in
polar coordinates, which are indeed complete in the space of all solutions in $L^2(\Omega_0,\R)$, where $\Omega_0=\{z\in\C\colon\ |z|<1\}$ denotes the unit disk in the
complex plane \cite{Vekua2013}. It follows directly from Proposition 
\ref{prop:f1f2} that all $F_n[\lambda](z)$ are
$\lambda$-metamonogenic. We also note that $F_n[\lambda]$ may be constructed
as
\begin{equation}
F_n[\lambda] = F_n^+[\lambda]\ei +  F_n^-[\lambda]\ej,
\end{equation}
in terms of the reduced-quaternionic valued functions  
\begin{align*}
   F^\pm_n[\lambda](z) &=  J_n(\lambda\rho )\Phi^\pm_n(z) -
   (\cos \theta \ei + \sin \theta\ej) J_n'(\lambda\rho )\Phi^\pm_n(z) \\
  &\quad\quad  \mp \frac{n}{\lambda\rho }(\sin \theta \ei - \cos \theta\ej)J_n(\lambda\rho )\Phi^\mp_n(z) ,
\end{align*}
where we write $\Phi^+_n(z)=\cos n\theta$,  $\Phi^-_n(z)=\sin n\theta$.
  

\subsection{Basic metamonogenics}

Next we introduce a special subset of the $\lambda$-metamonogenic
functions defined in the previous section. It is well known
\cite{JeffreyDai2008} that $J_n$ has a countable collection of
simple real zeros $j_{n,m}$,
\begin{equation*}
0< j_{n,1}< j_{n,2}< \cdots.
\end{equation*}

The basic metamonogenic functions are defined by
\begin{equation} \label{def:Fn}
  F_{n,m}(z) = F_n[j_{n,m}](z).
\end{equation}
for $n\ge0$, $m\ge1$.  Thus
\begin{equation} \label{eq:Fnmproperty}
  (D+j_{n,m})F_{n,m} = 0.
\end{equation}

Some examples of $F_{n,m}$ in $\Omega_0$ are given in Figure \ref{Figure-examples}. Our main result is as follows.
 \begin{figure}[t!]
  \pic{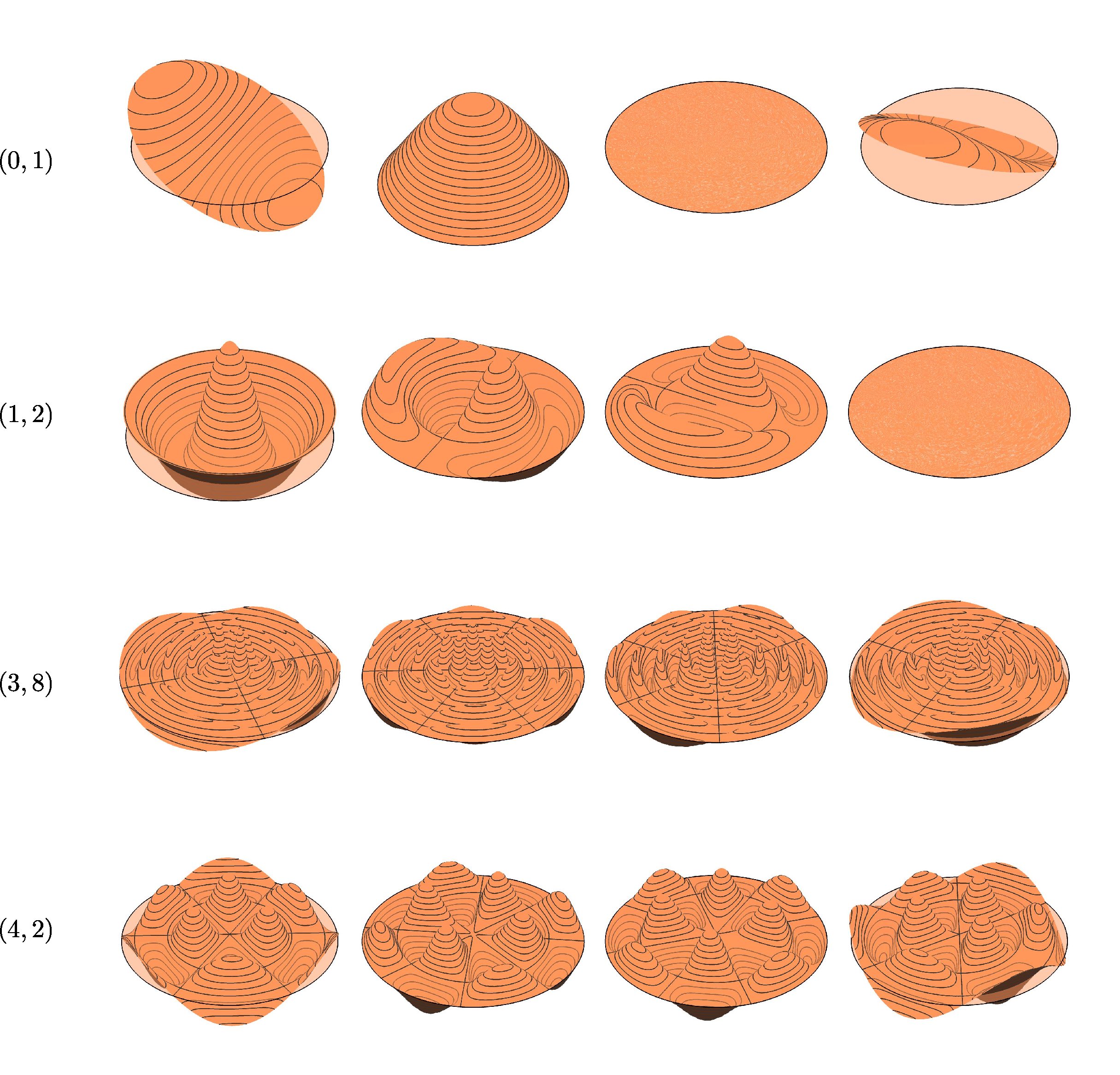}{.6,-17.0}{11}{scale=.16}
  \caption{The functions $F_{n,m}$ for assorted values of $(n, m)$. The scalar parts 
  are shown in the left column. The symmetries due to the presence of the functions $\Phi^{\pm}_n$ are clearly visible. \label{Figure-examples}}
\end{figure}

\begin{Theorem} \label{Main_theorem} (i) Let $(n_1,m_1)\not=(n_2,m_2)$.
  If $\{n_1,n_2\}\not=\{0,1\}$, then  
   \begin{align}    \label{eq:almostorthog1}
    \langle F_{n_1,m_1},\  F_{n_2,m_2} \rangle_{\H} = 0  ,
  \end{align}
while   
    \begin{align}
    \langle F_{0,m_1},\  F_{1,m_2} \rangle_{\H} &= -2\pi \, \frac{J_1(j_{0,m_1}) J_0(j_{1,m_2})}{j_{0,m_1}-j_{1,m_2}} \, \ei, \label{eq:almostorthog2} \\
    \langle F_{1,m_1},\  F_{0,m_2} \rangle_{\H} &= -2\pi \, \frac{J_0(j_{1,m_1}) J_1(j_{0,m_2})}{j_{1,m_1}-j_{0,m_2}} \, \ei. \label{eq:almostorthog22}
  \end{align}
 (ii)  The norms of these metamonogenic functions are given by
  \begin{equation}     \label{eq:norms}
    \| F_{n,m} \|^2_2 = 2\pi J^2_{n-1}(j_{n,m})
        =2\pi J^2_{n+1}(j_{n,m}).
  \end{equation}
  (iii) Let $f \in \M_2(\Omega_0;\lambda)$ where
  $\lambda\in\R\setminus\{0\}$. Then $f$ is in the closed subspace
  of the right quaternionic Hilbert module $L^2(\Omega_0)$ spanned by $\{F_{n,m}\colon n\ge0, m\ge1\}$; that is, there are $c_{n,m} \in \H$ such that
  \begin{equation*}
  f(z) = \sum_{n=0}^{\infty} \sum_{m=1}^{\infty} F_{n,m}(z) \, c_{n,m}.
  \end{equation*}
 \end{Theorem}

\begin{proof} 
  Since  $F_{n,m}$ is continuous in fact in the whole plane, it is clearly
  in $L^2(\Omega_0)$.   The proof divides naturally into parts.\\

\noindent\textit{(i) Orthogonality.} We must show that
\[ \int_{\Omega_0}  \overline{F_{n_1,m_1}(z)}\,F_{n_2,m_2}(z)\,dx\,dy =0
\]
whenever $(n_1,m_1)\not=(n_2,m_2)$ and $\{n_1,n_2\}\not=\{0,1\}$.  
We break down the integrand into quaternionic components as follows,
\begin{align*}
  &\overline{F_{n_1,m_1}(z)}\,F_{n_2,m_2}(z) \\
  &=
   ( A_1-B_1\ei-C_1\ej - D_1\ek )( A_2+B_2\ei+C_2\ej +D_2\ek ) \\
   &= A_1A_1 + B_1B_2 + C_1C_2 + D_1D_2 + (A_1B_2-B_1A_2-C_1D_2+D_1C_2)\ei \\
   &\quad + (A_1C_2 - C_1A_2 + B_1D_2 - D_1B_2)\ej + (-B_1C_2+C_1B_2+A_1D_2-D_1A_2)\ek.
\end{align*}
For convenience, let us write $j_1=j_{n_1,m_1}$,
$j_2=j_{n_2,m_2}$. With this notation, one finds after a great
deal of cancellation that 
\begin{align*}
  A_1A_2 &=  \big( J_{n_1}'(j_1 \rho) J_{n_2}'(j_2 \rho)  
    + \frac{n_2}{j_2\rho } J_{n_1}'(j_1\rho ) J_{n_2}(j_2 \rho)  
    + \frac{n_1}{j_1 \rho} J_{n_1}(j_1 \rho) J_{n_2}'(j_2 \rho) \\
         &\quad + \frac{n_1n_2}{j_1j_2 \rho^2} J_{n_1}(j_1 \rho) J_{n_2}(j_2 \rho)
            \big)\big(
        (\cos^2\theta)\Phi^+_{n_1}\Phi^+_{n_2}
     +  (\cos\theta\sin\theta)\Phi^+_{n_1}\Phi^-_{n_2} \\
     &\hspace{.38\textwidth}
     + (\sin\theta\cos\theta)\Phi^-_{n_1}\Phi^+_{n_2}
       + (\sin^2\theta) \Phi^-_{n_1}\Phi^-_{n_2} \big), \\[1ex]
  B_1B_2 &= J_{n_1}(j_1 \rho) J_{n_2}(j_2 \rho)\Phi^+_{n_1}\Phi^+_{n_2}, \\[1ex]
  C_1C_2 &= J_{n_1}(j_1 \rho) J_{n_2}(j_2 \rho)\Phi^-_{n_1}\Phi^-_{n_2}, \\[1ex]
  D_1D_2 &=   \big( J_{n_1}'(j_1 \rho) J_{n_2}'(j_2 \rho)  
    + \frac{n_2}{j_2\rho } J_{n_1}'(j_1 \rho) J_{n_2}(j_2 \rho)  
    + \frac{n_1}{j_1\rho } J_{n_1}(j_1 \rho) J_{n_2}'(j_2 \rho) \\
         &\quad + \frac{n_1n_2}{j_1j_2 \rho^2} J_{n_1}(j_1 \rho) J_{n_2}(j_2 \rho)
            \big)\big(
        (\sin^2\theta)\Phi^+_{n_1}\Phi^+_{n_2}
     - (\sin\theta\cos\theta)\Phi^+_{n_1}\Phi^-_{n_2} \\
     &\hspace{.38\textwidth}
     - (\cos\theta\sin\theta)\Phi^-_{n_1}\Phi^+_{n_2}
       + (\cos^2\theta) \Phi^-_{n_1}\Phi^-_{n_2} \big) .
\end{align*}
Now it is best to group the parts as follows: using $\Phi^+_{n_1}\Phi^+_{n_2} +
\Phi^-_{n_1}\Phi^-_{n_2}=\Phi^+_{n_1-n_2}$, first 
\begin{align*}
   A_1A_2 + D_1D_2 & = \big( J_{n_1}'(j_1 \rho) J_{n_2}'(j_1 \rho)  
    + \frac{n_2}{j_2\rho } J_{n_1}'(j_1\rho ) J_{n_2}(j_2 \rho) \\
          &\quad + \frac{n_1}{j_1 \rho} J_{n_1}(j_1 \rho) J_{n_2}'(j_2 \rho)
            + \frac{n_1n_2}{j_1j_2 \rho^2} J_{n_1}(j_1 \rho) J_{n_2}(j_2\rho )
            \big) \Phi^+_{n_1-n_2} \\
    B_1B_2+C_1C_2 &=   J_{n_1}(j_1 \rho) J_{n_2}(j_2 \rho) \Phi^+_{n_1-n_2} ,
\end{align*}
and then we can integrate,
\begin{align}
\langle F_{n_1,m_1},\  F_{n_2,m_2} \rangle_{\R} &= \nonumber
  \int_0^{2\pi}\int_0^1 (  A_1A_1 + B_1B_2 + C_1C_2 + D_1D_2 )\,\rho\,d\rho \,d\theta \\
  &=   \int_0^1 \big( \rho J_{n_1}'(j_1 \rho) J_{n_2}'(j_2 \rho)\nonumber
    + \rho J_{n_1}(j_1 \rho) J_{n_2}( j_2 \rho)    \\\nonumber
 &\quad \quad + \frac{n_2}{j_2}   J_{n_1}'(j_1 \rho) J_{n_2}(j_2 \rho) 
   + \frac{n_1}{j_1}  J_{n_1}(j_1 \rho) J_{n_2}'(j_2 \rho) \\ 
   &\quad\quad  + \frac{n_1n_2}{j_1j_2} \rho^{-1}
    J_{n_1}(j_1 \rho) J_{n_2}(j_2 \rho) \big) \,d\rho  
   \int_0^{2\pi} \Phi^+_{n_1\--n_2}\,d\theta. \label{eq:prod12}
\end{align}
For $n_1\not=n_2$ the $\theta$-integral is zero, and so is the scalar
product. Now suppose $n_1=n_2=n$, and consider the $\rho$-integral. We will use the 
following orthogonality property \cite{JeffreyDai2008} for Bessel functions
scaled by distinct zeros,
\begin{align} \label{eq:besselorthog}
  \int_0^1 J_n(j_{n,m_1} \rho) J_n(j_{n,m_2} \rho)\, \rho\,d\rho  =0
\end{align}
when $m_1\not=m_2$, as well as \eqref{eq:besselrecurrence1}.

Within the $\rho$-integral at the end of \eqref{eq:prod12} we find
\begin{align*}
  \int_0^1  \big( r J_n'(j_1 \rho) J_n'(j_2 \rho) +
   \frac{n^2}{j_1j_2} \rho^{-1}
    J_n(j_1 \rho) J_n(j_2 \rho) \big) \,d\rho ,
\end{align*}
which by \eqref{eq:besselrecurrence1} is equal to
\begin{align*}
  &=  \frac{1}{4}\int_0^1\big( ( J_{n-1}(j_1\rho )- J_{n+1}(j_1\rho )) 
  (J_{n-1}(j_2\rho )- J_{n+1}(j_2\rho ))  \\&\quad \quad+
  ( J_{n-1}(j_1\rho )+J_{n+1}(j_1\rho ))(
  J_{n-1}(j_2\rho )+ J_{n+1}(j_2\rho )) \big)\,\rho\,d\rho  \\
  &= \frac{1}{2}\int_0^1 (J_{n-1}(j_1\rho )J_{n-1}(j_2\rho )
  + J_{n+1}(j_1\rho )J_{n+1}(j_2\rho ) )\,\rho\,d\rho .
\end{align*}
Also in \eqref{eq:prod12} we have
\begin{align*}
  & \int_0^1\big( \frac{n}{j_2}   J_n'(j_1 \rho) J_n(j_2 \rho) 
   + \frac{n}{j_1}  J_n(j_1 \rho) J_n'(j_2 \rho) \big)   \, d\rho 
\\
  &\quad = \frac{1}{2}\int_0^1 \big( (J_{n-1}(j_1\rho )J_{n-1}(j_1\rho ) -
    J_{n+1}(j_1\rho )J_{n+1}(j_1\rho ) \big)  \,\rho\,d\rho ,
\end{align*}
again with the help of  \eqref{eq:besselrecurrence1}.
Combining these we find that
\begin{align} \label{eq:prodproof0}
\langle F_{n,m_1},\  F_{n,m_2} \rangle_{\R} &=
  2\pi \int_0^1\big( J_n(j_1\rho )J_n(j_2\rho ) +
   J_{n-1}(j_1\rho )J_{n-1}(j_2\rho ) \big)\,\rho\,d\rho ,
\end{align}
which is zero by \eqref{eq:besselorthog} as we are assuming $m_1\not=m_2$.

Similarly, straightforward computations yield that the vector part of
the quaternionic inner product is
\begin{align} \label{eq:prodproof1}
&{\rm Vec} \langle F_{n_1,m_1},\  F_{n_2,m_2} \rangle_{\H} \nonumber \\
   &= \left(\int_0^1\big( J_{n_1-1}(j_1\rho )J_{n_2}(j_2\rho ) -
   J_{n_1}(j_1\rho )J_{n_2-1}(j_2\rho) \big)\,\rho\,d\rho  
   \int_0^{2\pi} \Phi^+_{n_1+n_2-1}\,d\theta \right) \ei \nonumber \\
   &\quad + \left( \int_0^1\big( J_{n_1}(j_1\rho )J_{n_2}(j_2\rho ) +
   J_{n_1-1}(j_1\rho )J_{n_2-1}(j_2\rho ) \big)\,\rho\,d\rho  
   \int_0^{2\pi} \Phi^-_{n_1+n_2-1}\,d\theta \right) \ej \nonumber \\
   &\quad + \left(\int_0^1\big( J_{n_1-1}(j_1\rho )J_{n_2}(j_2\rho ) -
   J_{n_1}(j_1\rho )J_{n_2-1}(j_2\rho) \big)\,\rho\,d\rho  
   \int_0^{2\pi} \Phi^-_{n_1\--n_2}\,d\theta \right) \ek.
\end{align}
Analogous computations involving the remaining three components show that
${\rm Vec} \langle F_{n_1,m_1},\ F_{n_2,m_2} \rangle_{\H} = 0$
whenever $(n_1,m_1)\not=(n_2,m_2)$ and
$\{n_1,n_2\}\not=\{0,1\}$, which establishes \eqref{eq:almostorthog1}, while
\eqref{eq:almostorthog2}--\eqref{eq:almostorthog22} follow directly
from \eqref{eq:prodproof0}--\eqref{eq:prodproof1}.

\bigskip

\noindent\textit{(ii) Norms.} In \cite[p.\ 303]{JeffreyDai2008} we have
for any $\lambda$ that
\begin{align*}
  2 \int_{\rho_1}^{\rho_2} J_n(\lambda\rho )^2 \,\rho\,d\rho  &=
  \rho^2(J_n(\lambda\rho )^2-J_{n-1}(\lambda\rho )J_{n+1}(\lambda\rho ))\,|_{\rho_1}^{\rho_2} .
\end{align*}
Thus with $\lambda=j_{n,m}$, \eqref{eq:prodproof0}--\eqref{eq:prodproof1} specialize to
\begin{equation*}
\|F_{n,m}\|_2^2   = 2\pi \int_0^1 \big( \rho J_{n}(j_{n,m} \rho)^2
    + \rho J_{n-1}(j_{n,m} \rho)^2 \big) \,d\rho.
\end{equation*}
which gives \eqref{eq:norms}.

\vspace{1ex}
\noindent\textit{(iii) Completeness.} Now fix $\lambda$ and suppose that
$f=f_0+f_1\ei+f_2\ej+f_2\ek\in\M_2(\Omega_0;\lambda)$ is orthogonal to every
$F_{n,m}$ in the sense of \eqref{eq:InnerProduct}. Since every well defined function in $\Omega_0$ is periodic in $\theta$ in polar coordinates, it follows from Definition \ref{Definition_RQM functions} that
\begin{align}
 0 &= \langle  f, F_{n,m} \rangle_{\R} \nonumber \\
  &= \frac{1}{\lambda} \int_0^{2\pi}\!\!\!\int_0^{1}
    \left( \frac{\partial f_1}{\partial\rho } \, \rho \, J_n'(j_{n,m}\rho ) \Phi^+_{n} + \frac{n}{j_{n,m}} \frac{\partial f_1}{\partial\rho }J_n(j_{n,m}\rho ) \Phi^+_{n}
    - \frac{\partial f_1}{\partial \theta}J_n'(j_{n,m}\rho ) \Phi^-_{n} \right. \nonumber \\
  &\qquad \qquad \qquad - \frac{n}{j_{n,m}} \, \rho^{-1} \frac{\partial f_1}{\partial \theta}J_n(j_{n,m}\rho ) \Phi^-_{n} + \frac{\partial f_2}{\partial\rho } \, \rho \, J_n'(j_{n,m}\rho ) \Phi^-_{n} \nonumber \\
  &\qquad \qquad \qquad + \frac{n}{j_{n,m}} \frac{\partial f_2}{\partial\rho }J_n(j_{n,m}\rho ) \Phi^-_{n} +  \frac{\partial f_2}{\partial \theta}J_n'(j_{n,m}\rho ) \Phi^+_{n} \nonumber \\
  &\qquad \qquad \qquad \left. + \frac{n}{j_{n,m}} \, \rho^{-1} \frac{\partial f_2}{\partial \theta}J_n(j_{n,m}\rho ) \Phi^+_{n} \right) d\rho  \, d\theta \nonumber \\
  &\quad + \int_0^{2\pi}\!\!\!\int_0^{1} f_1 J_n(j_{n,m}\rho )\Phi^+_{n} \,\rho \, d\rho  \, d\theta + \int_0^{2\pi}\!\!\!\int_0^{1} f_2 J_n(j_{n,m}\rho )\Phi^-_{n} \,\rho \, d\rho  \, d\theta. \label{Integral_proof}
\end{align}
We apply integration by parts to the second and third terms of the first integral:
\begin{align*}
&\int_0^{2\pi}\!\!\!\int_0^{1} \frac{n}{j_{n,m}} \frac{\partial f_1}{\partial\rho }J_n(j_{n,m}\rho ) \Phi^+_{n} \, d\rho  \, d\theta
   - \int_0^{2\pi}\!\!\!\int_0^{1} \frac{\partial f_1}{\partial \theta}J_n'(j_{n,m}\rho ) \Phi^-_{n} \, d\rho  \, d\theta \\
&= \frac{1}{j_{n,m}} \int_0^{1} J_n(j_{n,m}\rho ) \left( \int_0^{2\pi} \frac{\partial f_1}{\partial\rho } ( \Phi^-_{n})' d\theta\right) d\rho  
- \int_0^{2\pi} \Phi^-_{n} \left( \int_0^{1} \frac{\partial f_1}{\partial\rho } J_n'(j_{n,m}\rho ) d\rho  \right) d\theta \\
&= -\frac{1}{j_{n,m}} \int_0^{2\pi}\!\!\!\int_0^{1} J_n(j_{n,m}\rho ) \frac{\partial^2 f_1}{\partial \theta \partial\rho } \Phi^-_{n} d\rho  \, d\theta 
+ \frac{1}{j_{n,m}} \int_0^{2\pi}\!\!\!\int_0^{1} J_n(j_{n,m}\rho ) \frac{\partial^2 f_1}{\partial\rho  \partial \theta} \Phi^-_{n} d\rho  \, d\theta \\
&= 0
\end{align*}
when $n > 0$ by \eqref{eq:bessel0}, and for the sixth and seventh terms,
\begin{equation*}
\int_0^{2\pi}\!\!\!\int_0^{1} \frac{n}{j_{n,m}} \frac{\partial f_2}{\partial\rho }J_n(j_{n,m} \rho ) \Phi^-_{n} \, d\rho  \, d\theta
   + \int_0^{2\pi}\!\!\!\int_0^{1} \frac{\partial f_2}{\partial \theta}J_n'(j_{n,m} \rho ) \Phi^+_{n} \, d\rho  \, d\theta = 0
\end{equation*}
also when $n > 0$. Integrating the remaining integrals by parts, we have
\begin{align*}
0 = \langle  f, F_{n,m} \rangle_{\R}
&= - \frac{1}{j_{n,m}^2} \int_0^{2\pi} \Phi^+_{n} \left( \int_0^{1} (\Delta_{\rho,\theta} f_1) J_n(j_{n,m} \rho ) \, \rho \, d\rho  \right) d\theta \\
   &\quad - \frac{1}{j_{n,m}^2} \int_0^{2\pi} \Phi^-_{n} \left( \int_0^{1} (\Delta_{r,\theta} f_2) J_n(j_{n,m} \rho ) \, \rho \, d\rho  \right) d\theta \\
  &\quad + \int_0^{2\pi}\!\!\!\int_0^{1} f_1 J_n(j_{n,m} \rho )\Phi^+_{n} \, \rho \, d\rho  \, d\theta \\
  &\quad+ \int_0^{2\pi}\!\!\!\int_0^{1} f_2 J_n(j_{n,m} \rho )\Phi^-_{n} \, \rho\, d\rho  \, d\theta
\end{align*}
when $n > 0$.

Since $f_1$, $f_2$ are $\lambda$-metaharmonic,
\begin{align*}
 (\lambda^2& + j_{n,m}^2) \left(\int_0^{2\pi}\!\!\!\int_0^{1} f_1 J_n(j_{n,m} \rho )\Phi^+_{n} \, \rho\, d\rho  \, d\theta + \int_0^{2\pi}\!\!\!\int_0^{1} f_2 J_n(j_{n,m} \rho )\Phi^-_{n} \, \rho\, d\rho  \, d\theta \right) \\
&= 0
\end{align*}
when $n > 0$. Similar arguments using the $\ei$, $\ej$, and $\ek$ components enable
one to show that in fact
\begin{align*}
  \int_0^{2\pi}\!\!\!\int_0^{1} f_1 J_n(j_{n,m} \rho)\Phi^-_{n} \, \rho \, d\rho \, d\theta &=0,\\
  \int_0^{2\pi}\!\!\!\int_0^{1} f_2 J_n(j_{n,m} \rho)\Phi^+_{n} \, \rho \, d\rho \, d\theta &= 0
\end{align*}
for $n>0$. By the completeness of the set $\{J_n(j_{n,m} \rho)\Phi^{\pm}_{n}\}$ in $L^2(\Omega_0,\R)$ it follows that $f_1$ and $f_2$ are in the linear span
of $\{J_0(j_{0,m} \rho)\Phi^{\pm}_{0}\}$. Since $\Phi^+_0=1$,  $\Phi^-_{0}=0$, we have the series representations
\begin{align}
  f_1 = \sum_{m=1}^\infty c_{1,m}J_0(j_{0,m}\rho), \quad
  f_2 = \sum_{m=1}^\infty c_{2,m}J_0(j_{0,m}\rho), \quad
\end{align}
converging in $L^2$ for real constants $c_{1,m}$, $c_{2,m}$. By
Proposition \ref{prop:f1f2},
\begin{align*}
  f_0 &= \frac{1}{\lambda}\bigg(\frac{\partial}{\partial x}
      \sum_{m=1}^\infty c_{1,m}J_0(j_{0,m}\rho) +
        \frac{\partial }{\partial y}  \sum_{m=1}^\infty c_{2,m}J_0(j_{0,m}\rho) \bigg)\\
    &= \sum_{m=1}^\infty \frac{j_{0,m}}{\lambda} J_0'(j_{0,m}\rho) (c_{1,m}\cos\theta+c_{2,m}\sin\theta),\\
  f_3 &= \frac{1}{\lambda} \bigg(\frac{\partial}{\partial y}
      \sum_{m=1}^\infty c_{1,m}J_0(j_{0,m}\rho) -
        \frac{\partial }{\partial x}  \sum_{m=1}^\infty c_{2,m}J_0(j_{0,m}\rho) \bigg)\\
    &= \sum_{m=1}^\infty \frac{j_{0,m}}{\lambda} J_0'(j_{0,m}\rho) (-c_{2,m}\cos\theta+c_{1,m}\sin\theta).
\end{align*}
Let $m'\ge1$. Using these series representations, first we look at the
scalar part of the hypothesis
\begin{align*}
  0 &= \langle F_{0,m'},\ f\rangle_\H \\
  &= \langle \ -J_1(j_{0,m'}\rho)\cos\theta+ J_0(j_{0,m'}\rho)\ei - J_1(j_{0,m'}\rho)\sin\theta \ek
  ,\\
   &\qquad \sum_{m=1}^\infty \frac{j_{0,m}}{\lambda} J_0'(j_{0,m}\rho) (c_{1,m}\cos\theta+c_{2,m}\sin\theta) + f_1\ei + f_2\ej \\
   &\qquad\qquad + \big(\sum_{m=1}^\infty \frac{j_{0,m}}{\lambda} J_0'(j_{0,m}\rho) (-c_{2,m}\cos\theta+c_{1,m}\sin\theta)\big) \ek \ \rangle_\H =0.
\end{align*}
By $L^2$ convergence, 
\begin{align*}
  0 &= -\sum_{m} c_{1,m} \frac{j_{0,m}}{\lambda} \int_0^{1} \rho J_1(j_{0,m'}\rho)J_0'(j_{0,m}\rho)\,d\rho
      \int_0^{2\pi} \cos^2\theta\,d\theta \\
  &\quad - \sum_{m}c_{2,m} \frac{j_{0,m}}{\lambda} \int_0^{1} \rho J_1(j_{0,m'}\rho)J_0'(j_{0,m}\rho)\,d\rho
    \int_0^{2\pi} \cos\theta\sin\theta\,d\theta \\
  &\quad -\int_0^{2\pi} \!\!\! \int_0^{1} J_0(j_{0,m'}\rho)f_1\,\rho\,d\rho d\theta \\
  &\quad -\sum_{m}c_{2,m} \frac{j_{0,m}}{\lambda} \int_0^1 \rho J_1(j_{0,m'}\rho)J_0'(j_{0m}\rho)\,d\rho
    \int_0^{2\pi} \sin\theta\cos\theta\,d\theta \\
  &\quad + \sum_{m} c_{1,m} \frac{j_{0,m}}{\lambda} \int_0^1\rho J_1(j_{0,m'}\rho)J_0'(j_{0m}\rho)\,d\rho
      \int_0^{2\pi}\sin^2\theta\,d\theta.
\end{align*}
Thus $f_1$ is orthogonal to $J_{0,m'}$ and hence is orthogonal in fact
to all $J_{n,m}\Phi^\pm_n$, which implies $f_1=0$. When one expands
the $\ek$ component of the inner product it is seen similarly that
$f_2=0$. In consequence, $f = 0$ identically as required.
\end{proof}

The information \eqref{eq:almostorthog2}--\eqref{eq:almostorthog22}
permits one to orthogonalize (say via the Gram-Schmidt process) the
subspace generated by $\{F_{0,m}, F_{1,m} \colon m\ge1\}$, which by
\eqref{eq:almostorthog1} will combine with the remaining $F_{n,m}$ to
give a full orthogonal basis. The resulting functions are not
particularly interesting, so we will omit the details.


\section{Time-dependent  solutions\label{sec:application}}

Consider the partial differential equation
\begin{align} \label{eq:timehelmholtz}
(\Delta + K^2\frac{\partial^2}{\partial t^2})v=0   
\end{align}
for $v(z,t)\in\H$, $z\in\Omega_0$, $t\ge0$. This can be interpreted
as a wave equation using imaginary time $it$.
(cf.\ the Wick transformation \cite{Burgess2003}).
 
We consider the natural quaternionic extensions of the real-valued
solutions of
\eqref{eq:timehelmholtz}. Since
\begin{align}\label{eq:timefactorize}
  \big(D + K\frac{\partial}{\partial t} \big )
  \big(D - K\frac{\partial}{\partial t} \big ) = -(\Delta^2+K^2),
 \end{align}
we are led to consider the companion equation
\begin{align}    \label{eq:timemt}
  \big(D + K\frac{\partial}{\partial t} \big )v=0 .
 \end{align}
 Since the operator $\Delta + K^2 (\partial^2/\partial t^2)$ has only
 real ingredients, it operates independently on each component of
 $v=v_0 +v_1\ei + v_2\ej + v_3\ek$.

\begin{figure}[b!]
 \pic{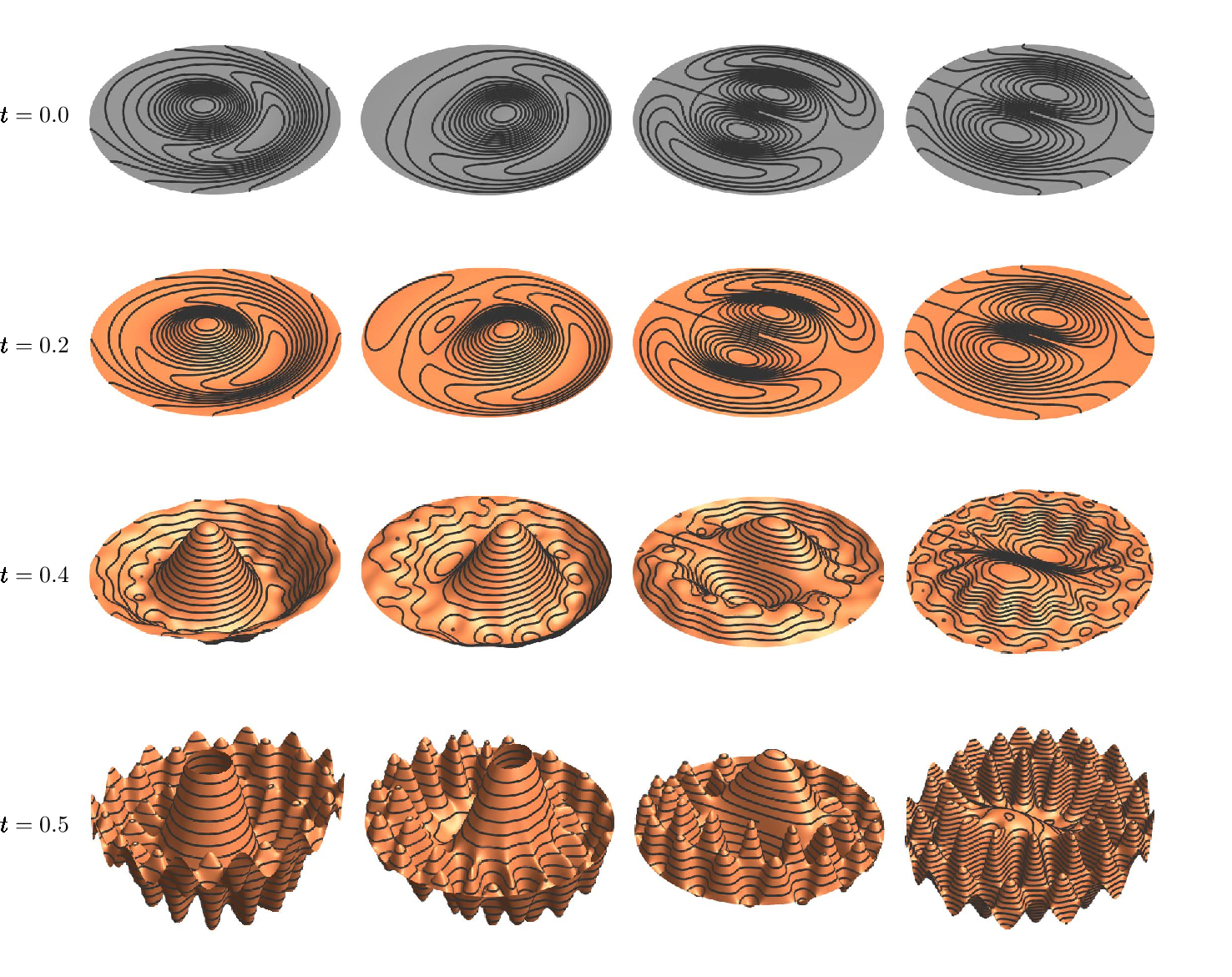}{.6,-15.}{9.6}{scale=.26}
 \caption{The initial condition ($t=0$, top row) contains high order terms which are not visible in the graphics until approximately $t>0.3$, when the exponential terms in time in \eqref{eq:totalwavefunction} become sufficiently large. \label{figwaves}}
\end{figure}

\def\uv{\underline{v}}
 
Because of \eqref{eq:Fnmproperty}, a time-dependent function given by
a series of the form
\begin{equation} \label{eq:totalwavefunction}
v(z,t) = \sum_{n=0}^{\infty} \sum_{m=1}^{\infty} F_{n,m}(z)\, c_{n,m} \, e^{j_{n,m} t}   
\end{equation} 
converging in $L^2(\Omega_0)$ clearly satisfies  \eqref{eq:timemt}.
One may propose a boundary value problem for this equation
with an initial condition given by an arbitrary $\uv(z) \in \M_2(\Omega_0;\lambda)$,
whose coefficients $c_{n,m}\in\H$ are given according to Theorem \ref{Main_theorem},
\begin{align}  \label{eq:vexpansion}
  \uv(z) = \sum_{n=0}^\infty\sum_{m=1}^\infty F_{n,m}(z) c_{n,m}. 
\end{align}
(In fact, one may prescribe only $\uv_0$ and $\uv_3$ in $\Ker(\Delta+\lambda^2)$
according to Proposition \ref{prop:f1f2}).
It is not difficult to show by means of the Cauchy-Kovalevskaya theorem
\cite{Petrovsky1954} that this is the only real-analytic solution of
\eqref{eq:timemt} satisfying the initial conditions $v(z,0)=\uv(z)$ and
\begin{align} \label{eq:vtime0}
   \left.\frac{\partial}{\partial t} \right|_{t=0} v_0(z,t) =
   \sum_{n=0}^\infty \sum_{m=1}^\infty j_{n,m} \, F_{n,m}(z) \, c_{n,m}.
\end{align}
(A similar result for reduced-quaternion-valued functions in elliptical
domains is worked out in detail in \cite{MoraisPorter2021}).

An example of the evolution of a wave function \eqref{eq:vexpansion}
is given in Figure \ref{figwaves}.
 

\section{Acknowledgments}

The first author's work was supported by the Asociaci\'on Mexicana de
Cultura, A.\ C.


 \newcommand{\authors}[1]{\textsc{#1}}
\newcommand{\booktitle}[1]{\textit{#1}}
\newcommand{\articletitle}[1]{``#1''}
\newcommand{\journalname}[1]{\textit{#1}}
\newcommand{\volnum}[1]{\textbf{#1}}

\end{document}